\documentclass[12pt]{amsart}
\usepackage{amssymb,amsmath,enumitem,color}
\usepackage{float}
\usepackage{graphicx}
\usepackage[linktocpage=true,colorlinks,citecolor=blue,
  linkcolor=blue,urlcolor=blue, pagebackref]{hyperref}
\usepackage{cite}
\usepackage{latexsym}
\usepackage{amscd}
\usepackage{amsthm}
\usepackage{mathrsfs}
\usepackage{url}
\usepackage{cleveref}
\usepackage{hyperref}
\usepackage[alphabetic,backrefs]{amsrefs}
\usepackage{lipsum} 
\usepackage{mathtools}
\usepackage{tikzpagenodes}
\usepackage{faktor}
\usetikzlibrary{tikzmark}
\usepackage[T1]{fontenc}
\usepackage{slashed}

\theoremstyle{plain}
\newtheorem{theorem}{Theorem}[section]
\newtheorem{cor}[theorem]{Corollary}
\newtheorem{lem}[theorem]{Lemma}
\newtheorem{prop}[theorem]{Proposition}

\newtheorem{rem}[theorem]{\it Remark}

\def\F{X}
\def\G{Y}
\def\U{p}
\def\V{n}

\def\2{NP$_2$}
\def\5{\sqrt5}
\def\6{\frac1\5}
\def\8{\infty}
\def\9{\frac{2\pi}3}
\def\GG{\mathbb{G}}
\def\Gr{\GG\mathrm{r}}
\def\+{\!+\!}
\def\-{\!-\!}
\def\con{\;\cong\;}
\def\npc{\mu}
\def\oS{\ol S}
\def\de{\delta}
\def\vp{\varphi}

\def\cH{\mathcal{H}}
\def\cD{\mathcal{D}}
\def\cL{\mathcal{L}}
\def\cR{\mathcal{R}}
\def\cO{\mathcal{O}}

\def\tg{h}
\def\tP{\tilde{P}}
\def\tr{\mathrm{tr}}

\def\Z{\mathbb{Z}}
\def\fS{\lower2pt\hbox{\large$\mathfrak{S}$}}
\def\dt{\frac\partial{\partial t}}

\def\diag{\mathrm{diag}}
\def\op{\oplus}
\def\ot{\otimes}

\def\so{\mathfrak{so}}
\def\na{\nabla}

\def\ka{\kappa}

\def\ol{\overline}
\def\ga{\gamma}
\def\la{\lambda}

\def\vt{\vartheta}
\def\vep{\varepsilon}
\def\La{\Lambda}
\def\Ga{\Gamma}

\def\Si{\Sigma}
\def\Th{\Theta}
\def\Up{\Upsilon}
\def\Ch{\raise2pt\hbox{$\chi$}}
\def\ow{\ol\w}
\def\oa{\ol a}
\def\oq{\ol q}
\def\oal{\ol\alpha}
\def\alp{\alpha}
\def\fm{\mathfrak{m}}
\def\w{\omega}
\def\W{\Omega}
\def\si{\sigma}

\def\half{{\ts\frac12}}

\def\C{\mathbb{C}}
\def\CP{\mathbb{CP}}
\def\H{\mathbb{H}}

\def\RP{\mathbb{RP}}
\def\R{\mathbb{R}}

\def\su{\mathfrak{su}}
\def\spin{\mathit{Spin}}

\def\ge{\geqslant}

\def\Im{\mathrm{Im}\,}
\def\Re{\mathrm{Re}\,}
\def\lra{\longrightarrow}

\def\lto{\leadsto}

\def\={\ =\ }
\def\y{\\[3pt]}
\def\yy{\\[7pt]}
\def\yyy{\\[15pt]}
\def\kr{\kern-1pt}
\def\ns{\kern-5pt}

\def\vs{\vskip10pt}

\def\we{\wedge}
\def\sul{\sum\limits}
\def\ba{\begin{array}}
\def\ea{\end{array}}
\def\be{\begin{equation}}
\def\bel#1{\be\label{#1}}
\def\ee{\end{equation}}
\def\col#1{\left(\kern-4pt\ba{l}#1\ea\kern-4pt\right)}
\def\ds{\displaystyle}
\def\ts{\textstyle}

\begin{document}
\parskip2pt
\parindent16pt
\mathsurround.5pt

\title[3-Sasakian and $G_2$-structures]
      {\large Revisiting 3-Sasakian and $G_2$-structures}

\author{Simon Salamon and Ragini Singhal}

\begin{abstract}
  The algebra of exterior differential forms on a regular 3-Sasakian
  7-manifold is investigated, with special reference to
  nearly-parallel $G_2$ 3-forms. This is applied to the study of
  3-forms invariant under cohomogeneity-one actions by $SO(4)$ on the
  7-sphere and on Berger's space $SO(5)/SO(3)$.
\end{abstract}

\maketitle

\vspace{-30pt}
{\small\tableofcontents}
\vspace{-30pt}

\section*{Introduction}

Let $M$ be an oriented self-dual Einstein 4-manifold $M$ with positive
scalar curvature. The bundle of 2-forms \bel{+-}
\La^2T^*M=\La^+\op\La^-\ee splits into self-dual (SD) and
anti-self-dual (ASD) subbundles, in line with the double covering \[
SO(4)\to SO(3)^+\times SO(3)^-\] of Lie groups. Our hypothesis implies
that $\La^-$ has constant curvature in the sense of \eqref{PW}
below. The reason for focussing on $\La^-$ is that (with its standard
orientation and Fubini-Study metric) the projective plane $\CP^2$
satisfies the hypothesis. Much of the theory extends to orbifolds, and
sources of examples include quaternion-K\"ahler quotients
\cite{GL,De1}, and the $SO(3)$-invariant orbifolds discovered by
Hitchin \cite{H}. The latter are parametrized by an integer $k\ge3$
and the first two examples are $S^4$ and $\CP^2/\Z_2$.

The (total space of the) $SO(3)^-$ principal bundle $P^-$ associated
to $\La^-$ is 3-Sasakian \cite{K}, and therefore admits an Einstein
metric and a family of nearly-parallel $G_2$ (for short, \2)
structures. An equivalent way of expressing the 3-Sasakian condition
is to say that the Riemannian cone $P^-\!\times\R^+$ is hyperk\"ahler,
meaning that it has a metric with holonomy $Sp(2)$
\cite{Bae}. However, it is well known that a squashing of the fibres
of $P^-$ produces a second Einstein metric and canonical \2 structure
whose cone has holonomy \emph{equal} to $\spin(7)$.

Now suppose (restricting to an open set if necessary) that $M$ is
spin, so that its structure group lifts to $\spin(4)\cong SU(2)^2$.
The negative spin bundle $\Si^-$ can be regarded as a quaternionic
line bundle over $M,$ and admits a family of metrics with $\spin(7)$
holonomy, all asymptotic to the cone over (a $2:1$ cover of) the
squashed \2 manifold \cite{BS}. An initial purpose of this article is
to relate well-known constructions of the relevant differential forms
on these spaces, one via a $\spin(7)$ 4-form on a vector bundle, the
other via 3-Sasakian 1-forms on a principal bundle.
This is motivated by the situation in which $M$ itself admits an
action by $SO(3)$ with 3-dimensional orbits. In this case, $P^-$ or
its double cover $\tP^-$ will admit an action by $SO(4)$ or $SU(2)^2$
of cohomogeneity one, preserving the $G_2$ structures. This happens
when $M$ is $S^4$ or $\CP^2,$ and more generally for the Hitchin
orbifolds.

We shall make this explicit for the 7-sphere $S^7\cong\tP^-\!S^4,$ in
order to express its squashed \2 structure in canonical form. The
result is consistent with a known expression for the round metric on
$S^7$ subject to $SO(4)$ and $\Z_3$ symmetry \cite{Z}. The same
technique is applied to the isotropy irreducible space
$B^7=SO(5)/SO(3),$ which Berger showed has positive sectional
curvature \cite{B}. This space also has a homogeneous \2 structure
\cite{S}, a cohomogeneity-one action by $SO(4)$ \cite{PV}, and is
diffeomorphic to $\tP^+$ for a non-trivial Hitchin orbifold
\cite{GWZ}. Although not our main concern, exhaustive classifications
of associative submanifolds in $B^7$ and in squashed 3-Sasakian
manifolds are undertaken in \cite{BM2,BM1} using twistor
correspondences, pursuing work in \cite{L,K} for $S^7$.

Paying debts to this previous mountain of work, our survey highlights
the distinction between the homogeneous \2 structures on $S^7$ and
$B^7$, when expressed in terms of the local differential geometry on
an interval times $SO(4)$. The choice of geodesics orthogonal to the
orbits allows one to construct bases $(\F_i)$ and $(\G_i)$ of
invariant 1-forms that enable the respective 3-forms to be built up
using upper triangular $2\times2$ matrices; see Corollaries \ref{FF}
and \ref{GG}. These bases ultimately arise from different conjugacy
classes of 3-dimensional subgroups of $SO(4)$, and different group
diagrams for the cohomogeneity-one actions.

Just as a hypersurface of a $G_2$ holonomy manifold acquires a
so-called half-flat $SU(3)$ structure, so a hypersurface of a \2
manifold has an analogous nearly-half-flat (NHF) property. There is a
sense in which a nearly-K\"ahler manifold has both properties. We
conclude with a brief investigation of induced structures on
hypersurfaces of $S^7$ and $B^7$, and we show that constants defining
the homogeneous \2 structures are to some extent determined by the NHF
condition. The results of this section form part of a broader project
to describe exceptional structures in dimensions 6 and 7.

\section{Notation}\label{not}
 
In the first part of this paper, we shall adopt the technique of \cite{BS},
which relies on quaternion-valued differential forms on the total
spaces of $\La^-$ and $\Si^-,$ though only the latter will concern
us.

The Riemannian metric of $M^4$ will be denoted by $g_4$ and its positive
scalar curvature by $12\ka,$ where $\ka$ is another constant. The
sphere $S^n(r)$ of radius $r$ has scalar curvature $n(n-1)/r^2,$ so
$\ka=1$ for $S^4(1)$. {\bf From now on, we shall scale $g_4$
  so that $\ka=1$.}

The induced connection on $\Si^-$ is defined by a vector-valued 1-form
\[ \phi = i\phi_1+j\phi_2+k\phi_3\]
on the total space of $\Si^-$. We shall refer to a differential form
with values in $\Im\H$ as `vector valued', as opposed to
`quaternion-valued' (for $\H$). Thus, $\bar\phi=-\phi,$ where a bar
denotes quaternionic conjugate. The curvature of $\Si^-$ is
\bel{Phi}\ba{rcl} \Phi &=& d\phi+\phi\we\phi\y &=&
i(d\phi_1+2\phi_{23})+i(d\phi_2+2\phi_{31})+i(d\phi_3+2\phi_{12})
\ea\ee where $\phi_{23}=\phi_2\we\phi_3$. See \cite[p\,841]{BS}.

The choice of a unit section of $\Si^-$ determines a local orthonormal
basis of ASD 2-forms via the isomorphism $S^2\Si^-\con(\La^2_-)_c$ of
complex vector spaces or bundles.  These 2-forms are the
components of the tautological vector-valued 2-form
\[ \half(\ow\we\w)=i\W_1+i\W_2+k\W_3=\W.\]
We use a Greek letter without an index to stand for a quantity with
values in $\H$ that encodes all the individual components.
Scalings are fixed by the existence pointwise of an orthonormal basis
$\{\w_0,\w_1,\w_2,\w_3\}$ of 1-forms for $g_4$ such that
\bel{0123} \W_k \= \w_{0k}-\w_{ij}\ee
for $(ijk)=(123)$.

It follows from \eqref{Phi} that \bel{dW} d\Phi=-2\phi\we\Phi\ee where
$\we$ combines multiplication of imaginary quaternions (cross product)
with wedging. With the notation above, the Einstein ASD condition
becomes \bel{PW} \Phi=\half\ka \W\ee as stated in \cite[p\,842]{BS},
although we are assuming throughout that $\ka=1$.

Now let
\bel{a} a=a_0+a_1i+a_2j+a_3k\ee
be a quaternionic coordinate on the fibre of $\Si^-$.  Note that
$\oa a=a\oa=R$ is the radius squared (denoted $r$ in
\cite[p\,846]{BS}). Set \bel{alp} \alp = da-a\phi =
\alp_0+i\alp_1+j\alp_2+k\alp_3.\ee At each point, the annihilator of
$\alp$ is the horizontal subspace determined by the connection. We can
then determine the effect of the action $a\lto aq$ on differential
forms. It yields the transformations $\phi\lto\oq\phi q$ and
$\W\lto\oq\W q$.

Next, consider the quaternion-valued 1-form
\bel{B} \beta\ =\ \oa\alp \= \beta_0+i\beta_1+j\beta_2+k\beta_3\ee
as defined in \cite[p\,846]{BS}. In particular,
\[\beta_1 \= a_0\alp_1-a_1\alp_0-a_2\alp_3+a_3\alp_2\]
is the result of applying a fibrewise ASD almost complex structure to
$\beta_0=\half dR$. However, from now one, we work on $\tP^-,$ so
that $R=1,$ and $\beta$ becomes a vector-valued 1-form. In this setting,\vs

\begin{lem}\label{dbeta}
$d\beta + \beta\we\beta = - 2\,\Im(\phi\we\beta) - \Phi$.
\end{lem}

\begin{proof} Using \eqref{Phi}, \eqref{PW}, and \eqref{alp},
\bel{dalp} d\alp
\= -(\alp+a\phi)\we\phi -a d\phi\y
\= -\alp\we\phi - a\Phi.\ee
Therefore 
\[\ba{rcl} d\beta\ =\ d(\oa\alp)
&=& d\oa\we\alp+\oa\,d\alp\y
&=& d\oa\we\alp+\oa(-\alp\we\phi-a\Phi)\y
&=& \oal\we\alp -(\phi\we\beta+\beta\we \phi)-\Phi.
\ea\]
Now restrict to $R=1,$ so that $\beta$ (and, as always, $\phi$) have values
in $\Im\H$. Because of the wedging, the real parts of $\phi\we\beta$ and
$\beta\we\phi$ cancel out. Moreover,
\bel{BB} \oal\we\alp = \oal a\we\oa\alp = \ol
\beta\we\beta=-\beta\we\beta\ee and the result follows.
\end{proof}

The first component of $\beta$ is \bel{dbeta1}
d\beta_1+2\beta_{23} \=
-2(\phi_2\we\beta_3-\phi_3\we\beta_2)-\Phi_1.\ee Note that
$\beta_{23}$ is totally vertical whereas $\Phi_1=\half\ka\W_1$ `lives'
on $M,$ so their exterior derivatives cannot be proportional, which
forces the appearance of the mixed terms on the right-hand side. The
fibrewise ASD property of $d\beta$ is indicated by the presence of
$\oal\we\alp$ rather than the self-dual $\alp\we\oal$.\vs

\section{Spin(7) and $G_2$-structures}

One of the main results in \cite[\S4]{BS}, by now well known, is that
the total space of $\Si^-$ admits a family of metrics with holonomy
equal to $\spin(7)$ and asymptotically conical. The induced Einstein
metric $\tg_7$ on the link of this conical limit therefore underlies a
so-called nearly-parallel $G_2$ structure (our abbreviation is
\2). The latter comprises a 3-form $\vp$ with stabilizer $G_2$ for
which
\bel{NP} d\vp=\npc*\vp\ee
where $21\mu^2/8$ is known to equal the scalar curvature of $\tg_7$
\cite{B}. The metric $\tg_7$ is also said to have weak holonomy $G_2$
\cite{Gr}. We proceed to write down $\vp$ explicitly.

In order to extract useful real-valued differential forms from those
defined in Section 1, we need to exploit the obvious invariants for
the action of $SO(3)$ on $\Im\H\cong\R^3,$ namely inner product and
determinant. In this way, we can define the real 3-forms
\[\ba{rclcl}
\Up &=& \beta_{123} &=& -\frac16\,\beta\we\beta\we\beta\y
\Th &=& \Phi_1\we\beta_1+\Phi_2\we\beta_2+\Phi_3\we\beta_3
    &=& -\Re[\Phi\we\beta]
\ea\]
and the real 4-forms
\[\ba{rclcl}
-d\Up &=& \Phi_1\we\beta_{23}+\Phi_2\we\beta_{31}+\Phi_3\we\beta_{12}
    &=& -\half\Re[\Phi\we\beta\we\beta]\y
\Ch &=& \Phi_1\we\Phi_1+\Phi_2\we\Phi_2+\Phi_3\we\Phi_3
    &=& -\Re[\Phi\we\Phi].
\ea\]
Using these, we can build $G_2$ structures. In particular:\vs

\begin{prop}\label{2.1}
  Let $\la=2/\5$. The 3-form \bel{varphi}\ts \vp \= 2\la\,\Th +
  \la^3\Up\ee defines an \2 structure on $\tP^-$ with $\npc=6/\5$. The
  associated Einstein metric is \bel{gt7}\ts \tg_7 \=
  \frac45\sul_{l=1}^3\beta_i\ot\beta_i + \pi^*g_4.\ee
\end{prop}

\begin{proof}
The calculation of $d\Th$ is a little delicate:
\[\ba{rcl} d\Th
&=& -\Re[d\Phi\we\beta+\Phi\we d\beta]\y
&=& -\Re[-2\phi\we\Phi\we\beta]
    -\Re[\Phi\we(-\beta\we\beta-2\phi\we\beta-\Phi)]\y
&=& 2\,d\Up-\Ch.
\ea\]
Terms in $\phi\we\Phi$ and $\Phi\we\phi$ cancel out because wedging is
commutative but imaginary quaternion multiplication is not.

Recall from \eqref{0123} that $\Phi_i=\half(\w_{0i}-\w_{jk})$ is a
horizontal ASD 2-form, so $\Ch=-\frac32\w_{0123},$ the volume form of
$(M,g_4)$ being $\w_{0123}$. It follows that \eqref{varphi} defines a
$G_2$ structure whose associated metric has an orthonormal basis
\bel{ON}
\w_0,\ \ \w_1,\ \ \w_2,\ \ \w_3,\ \
\la\beta_1,\ \ \la\beta_2,\ \ \la\beta_3\ee
so
\[\ts *\vp \= 2\la^2d\Up-\frac23\Ch.\]
On the other hand,
\[d\vp \= \la(4+\la^2)d\Up - 2\la\Ch\]
so the \2 conditin is satisfied when
\[\npc = \half\la^{-1}(4+\la^2) = 3\la\]
and $\la^2=4/5$. The sign of $\la$ and $\npc$ is merely a matter of
orientation, so we choose them to be positive for this proposition.
The form of $\tg_7$ follows from \eqref{ON}.\end{proof}

The complete metric $g_\Phi$ with holonomy $\spin(7)$ appears on
\cite[p\,848]{BS}. If we take the asymptotic limit and set $r=1,$ we
obtain $\frac15\tg_7$. It is well known that \eqref{gt7} is the
so-called squashed Einstein metric that exists on the $S^3$ (or
$\RP^3$) bundle over $M^4$. The more standard Einstein metric (the
round metric in the case of $\pi\colon S^7\to S^4$) has the same form,
but without the factor $1/5$.

The proof above is modelled on \cite[Prop.~2.5]{GS} from the theory of
3-Sasakian 7-manifolds. Indeed, we shall explain next that the 3-form
\eqref{varphi} coincides with one that can be defined using
globally-defined 1-forms $\eta_1,\eta_2,\eta_3$ dual to the Lie
algebra of $SU(2)$ acting on the fibre of a regular 3-Sasakian
7-manifold. This is accomplished by means of a gauge transformation
that converts the geometry of a vector bundle to that of a principal
bundle.

Modify \eqref{B} by conjugation and (partly for convention) a sign
change: \bel{eta} \eta \= -a\beta\oa \= -\alp\oa=a\oal\ee where $a$ is
a unit quaternion. Retaining the condition $R=1,$ we can write
\[\eta \= i\eta_1+j\eta_2+k\eta_3.\]
In particular,
\[ \eta_1 \= -a_0\alp_1+a_1\alp_0-a_2\alp_3+a_3\alp_2.\] 
Observe that, as a direct consequence of the conjugation, $\eta$ is
{\it invariant} under the action $a\lto aq$ of $SU(2)^-$. Moreover,
$\eta,$ like $\beta,$ annihilates the horizontal distribution on
$\tP^-$. It follows that each real component $\eta_i$ is globally
defined on $\tP^-$. Moreover,
\[ \alp\we\oal = \alp\oa\we a\oal = -\eta\we \eta\]
in analogy to \eqref{BB}, but with a fibrewise self-dual
property. Converting from $\beta$ to $\eta$ furnishes a more succinct
version of Lemma 1.1 on $\tP^-,$ namely\vs

\begin{lem}\label{dH} $d\eta + \eta\we \eta = a\Phi\oa$.\end{lem}

\begin{proof} Applying \eqref{alp} and \eqref{dalp},
\[\ba{rcl} d\eta\ =\ -d(\alp\oa)
&=& -d\alp\,\oa +\alp\we d\oa\y
&=& (\alp\we\phi+a\Phi)\oa+\alp\we(\oal-\phi\oa)\y
&=& \alp\we\oal+a\Phi\oa.
\ea\]
As has to be, the individual terms in the last line are unchanged when
$a,\alp,\Phi$ are replaced by $aq,\alp q,\oq\Phi q,$ respectively.
\end{proof}

It is instructive to compare Lemmas 1.1 and 2.2. In analogy to
\eqref{Phi}, set \bel{H} H\= a\Phi\oa\=
i(d\eta_1+2\eta_{23})+i(d\eta_2+2\eta_{31})+i(d\eta_3+2\eta_{12}).\ee
Then $H$ is `3-horizontal', i.e.\ its individual real components are
pointwise the pullbacks of forms on the 4-manifold $M$. However, they
are not invariant by $SU(2)^-,$ only their quaternionic combination
enjoys this property, the price to pay for avoiding analogues of the
terms $\phi\we\beta$ in Lemma 1.1. If we replace the symbol $H_i$ by
$\ol\phi^i$ (the overline is just notation), this coincides with
\cite[p\,50]{GS}.\vs

\begin{rem}\label{S20}\rm
  Conjugation of an imaginary quaternion by the
unit quaternion \eqref{a} describes the double covering $SU(2)^-\lra
SO(3)^-$ by mapping \eqref{a} to the orthogonal matrix
\bel{P} (P_{i,j})=\left(\ba{ccc} a_0^2+a_1^2-a_2^2-a_3^2 &
2(-a_0a_3+a_1a_2) & 2(a_0a_2+a_3a_1)\y 2(a_0a_3+a_1a_2) &
a_0^2-a_1^2+a_2^2-a_3^2 & 2(-a_0a_1+a_2a_3)\y 2(-a_0a_2+a_3a_1) &
2(a_0a_1+a_2a_3) & a_0^2-a_1^2-a_2^2+a_3^2 \ea\right)\ee
that features in many articles on the quaternions, such as
\cite[Prop.~23.12]{ASG}. The nine entries of \eqref{P} are tracefree
quadratics, elements of $S^2_0(\R^4)\cong\La^2_+\ot\La^2_-$.  Indeed,
$P_{i,j}$ equals minus the contraction of standard basis elements of
$\La^+$ and $\La^-$. The triple $(P_{1,i})$ is the hyperk\"ahler
moment map $\R^4\to\so(3)$ for the action of $U(1)$ on $\R^4$ tangent
to $I,$ and is familiar from the Gibbons-Hawking construction. In
components, \bel{deta1}\ts d\eta_1 + 2\eta_{23} = 2\sul_pP_{1,p}\W_p.\ee
\end{rem}

It is easy to compute
$dH$ using \eqref{H} twice. In particular,
\[ dH_1
\= 2(d\eta_2\we\eta_3-\eta_2\we d\eta_3)
\= -2(\eta_2\we H_3-\eta_3\we H_2)\]
which (as it has to) matches \eqref{dW}:
\[ d\Phi_1 \= -2(\phi_2\we\Phi_3-\phi_3\we\Phi_2)\]
All the factors of the right-hand side of the last line are
3-horizontal.\vs

\begin{prop}\label{2.2} Let $\la=2/\5$. Define
\[\ts\vp' \= 2\la\sul_{l=1}^3H_i\we\eta_i + \la^3\eta_{123}.\]
Then $\vp'$ coincides with minus the 3-form \eqref{varphi}, and
therefore gives an equivalent description of the squashed \2 structure
on $\tP^-$.
\end{prop}

\begin{proof} The definitions of $\eta$ and $H,$ and some telescoping,
imply that
\[ \vp' = -a\vp\oa = -\vp\]
since both $\vp$ and $\vp'$ are real differential forms.\end{proof}

The definition of $\varphi'$ in Proposition \ref{2.2} is modelled on
that of $\vp,$ but is perhaps more familiar. However, there is a
significant difference between Lemmas 1.1 and 2.2, either of which can
be used to compute $d\vp'=-d\vp$. The necessity to change the sign of
$\la$ arises from a minus sign that we inserted in the definition
\eqref{eta} of $\eta$ to preserve the combination $d\eta+\eta^2$ to
match $d\beta+\beta^2$.

As explained in \cite[p\,51]{GS}, in the 3-Sasakian context, one can
define a hyperk\"ahler triple of exact 2-forms \bel{hk} \varpi_i =
d(2R\eta_i) = 4\eta_{0i}-4R\eta_{jk}+2RH_i\ee on the cone over
$\tP^-$. Pointwise, each $2H_i=\w_{0i}-\w_{jk}$ is a horizontal ASD
2-form that matches the other terms, see \eqref{ON} with $\la=2$. The
existence of such a hyperk\"ahler structure is well known in the
setting of twistor space and quaternionic geometry \cite{Sw}, though
we are here more interested in induced $G_2$ structures on the
3-sphere bundle $\tP^-$.

The metric associated to \eqref{hk} is the cone over $X,$ and this is
one way of defining a 3-Sasakian structure \cite{Bae,BG}. Since the
holonomy of the cone lies in $Sp(2),$ it also lies in $\spin(7)$ and
$\tP^-$ must have an \2 structure. However, such a subgroup $\spin(7)$
is conjugate to the stabilizer of a 4-form
$-\varpi_1{}\!^2+\varpi_2{}\!^2+\varpi_3{}\!^2,$ so one needs to break
the $SO(3)$ symmetry of the hyperk\"ahler triple, whereupon there
emerges a family of \2 structures.

We can replace each $\beta_i$ by $\eta_i$ in \eqref{gt7} without
affecting the sum. More general families of metrics of the form
  \[\ts dt^2+\sum\limits_{k=1}^3A_i(t)^2\eta_i\ot\eta_i+B(t)^2g_4,\]
 with holonomy $\spin(7)$ were constructed by Bazaikin by deforming
 3-Sasakian structures \cite{Baz}. This has guided our own presentation
 of (or rather, hope to generalize) the homogeneous \2 structures, in
 which trigonometric functions of $t$ appear on subspaces
 $\left<\beta_k,\w_k\right>$ for $k=1,2,3$, see Corollaries \ref{FF}
 and \ref{GG}.

\section{The cohomogeneity-one action of SO(4) on S$^7$}

The group $Sp(n)$ acts transitively on the sphere $S^{4n-1}$ in
quaternionic space $\H^n$ and, in particular, \bel{S7}
S^7\con\frac{Sp(2)}{Sp(1)}\con\frac{SO(5)}{SU(2)}.\ee This
is the homogeneous model of 3-Sasakian geometry in dimension 7, with
its triple of invariant 1-forms.
From the point of view of Section 2, it is more appropriate to start
from the description \bel{SS} S^7\con\frac{Sp(2)\times
  Sp(1)}{Sp(1)\times Sp(1)} \= \frac{Sp(2)Sp(1)}{Sp(1)Sp(1)},\ee which
features in the classification \cite{FKMS} of homogeneous \2
spaces. The second factor of the isotropy group $Sp(1)Sp(1)\cong
SO(4)$ acts diagonally, and juxtaposition indicates a quotient by
$\pm1$ across both factors.  In quaternion-K\"ahler jargon, $S^7$ is
here regarded as the unit sphere in the $Sp(2)Sp(1)$ module $\fm$,
with complexification $\fm_c=E\ot H,$ where $Sp(2)$ acts on
$E\cong\C^4$ and $Sp(1)$ acts on $H\cong\C^2$.

Let $\Si^+\op\Si^-$ denote the spin representation of $\spin(4)\cong
Sp(1)\times Sp(1),$ so that $Sp(1)^+$ acts on $\Si^+\cong \C^2$ and
(as before) $Sp(1)^-$ acts on $\Si^-\cong H$. The isotopy subgroup in
\eqref{SS} then decomposes $E=\Si^+\op\Si^-,$ so that
\[ \fm_c\con(\Si^+\op\Si^-)\ot\Si^-.\]
Since $S^2\Si^-$ is the complexification of the space $\La^-$ of ASD
2-forms on $\R^4,$ the isotropy subgroup $SO(4)$ acts on the real
tangent space to $S^7$ as $\R^4\op\La^-$.  It is well known that this
action factors through $G_2$. Moreover there are independent $SO(4)$
invariant 3-forms $\vp_1,\vp_2$ on $S^7,$ a linear combination $\vp$
of which must satisfy \eqref{NP} with $\npc=6/\5,$ in accordance with
Proposition \ref{2.1}.

To obtain an action by $SO(4)\cong Sp(1)^+Sp(1)^-$ of cohomogeneity
one on $S^7,$ we move from isotropy subgroups to isometry groups and
set $E=S^3\Si^+,$ so that
\[  \fm_c \con S^3\Si^+\ot\Si^-\]
is the isotropy represention of the Wolf space $G_2/SO(4),$ as in
\cite{Z}. Then

\bel{9}\ba{rcl} S^2\fm_c
&\cong& (S^2E\ot S^2H)\op(\La^2E\ot\La^2H)\y &\cong& (S^6\Si^+\ot
S^2\Si^-)\op(S^2\Si^+\ot S^2\Si^-)\op S^4\Si^+\op\C.\ea\ee
The penultimate summand is the complexification of the space
$S^2_0(\R^3)\cong\R^5$ of $3\times3$ tracefree symmetric matrices.
Projection to this module of $v\ot v$ (where $v$ is a non-zero vector
in $\fm=\R^8$) defines the $Sp(2)Sp(1)$-equivariant fibration $S^7\to
S^4$.\vs

\begin{rem}\rm
The 9-dimensional summand of \eqref{9} is isomorphic to $S^2_0(\R^4)$
(cf.\ Remark \ref{S20}), and is the model for a geometrical stuctures
on 9-manifolds investigated in \cite{FN}. In the present context, it
gives rise to an embedding of $\RP^7$ as a hypersurface of $S^8$.
\end{rem}

Let $A\in S^4\subset S^2_0(\R^3)$ be a $3\times3$ matrix with distinct
eigenvalues. Then the orbit of $SO(3)=SU(2)^+/\Z_2$ containing $A$ has
dimension 3 and principal isotropy group $\Z_2^2$. There are
two singular orbits, each copies of $\RP^2,$ corresponding to
repeated eigenvalues, which can be negative or positive.
To study the round metric $g_4$ on $S^4$ relative to the action by
$SO(3)$ of cohomogeneity one, set $\vt_k=2\pi(k-1)/3,$ and
\bel{pcs}\ba{l}
c_k=c_k(t)=\cos(t+\vt_k)\y
s_k=s_k(t)=\sin(t+\vt_k)\ea\ee
for $k=1,2,3$. These trigonometric quantities satisfy
\bel{cs} c_1+c_2+c_3\=0\=s_1+s_2+s_3\ee
as well as the identity
\bel{ccss} c_jc_k+s_js_k=-\half,\ee
for $(ijk)=(123)$.

Now choose left-invariant 1-forms on $SO(3)$ satisfying \bel{eee}
de_1=-2e_{23},\quad de_2=-2e_{31},\quad de_3=-2e_{12}.\ee This
convention is chosen to match expressions such as
\eqref{dbeta1}. Combined with the constraint on $\ka,$ it leads us to
adopt an orthonormal basis of 1-forms on $S^4$ formed from $dt$ and
$4s_ke^k$ ($k=1,2,3$), so that \bel{g4}\ts g_4 =
dt^2+16\sul_{k=1}^3s_k(t)^2e_k\ot e_k\ee away from the singular
orbits. The associated basis of ASD 2-forms consists of
\[\ba{l}
\w_1 = 4s_1dt\we e_1-16s_2s_3 e_{23}\y
\w_2 = 4s_2dt\we e_2-16s_3s_1 e_{31}\y
\w_3 = 4s_3dt\we e_3-16s_1s_2 e_{12}.
\ea\]\vs

\begin{lem}\label{trick}
With this choice, the connection 1-form $\phi$ for $\La^2_-T^*\!S^4$ is
given by
\bel{phi} -\phi=i(2c_2+1)e_1+j(2c_2+1)e_2+k(2c_3+1)e_3,\ee
and $\ka=1$.
\end{lem}

\begin{proof} It suffices to verify \eqref{dW} and \eqref{PW}, since
these equations have a unique solution by the theories of Cartan or
Levi-Civita.

First observe that
\[d\w_1 \= 8(s_1-2(c_2s_3+s_2c_3))dt\we e_{23},\]
whereas
\[\ba{rcl}
-2(\phi_2\we\w_3-\phi_3\we\w_2)
&=& 8(ac_2-1)s_3 dt\we e_{23}+8(ac_3-1)s_2dt\we e_{23}\y
&=& 8(s_1-2(c_2s_3+c_3s_2))dt\we e_{23}.\ea\]
This verifies \eqref{dW} after cyclic permutations.

As regards \eqref{PW}, we have
\[\ba{rcl} d\phi_1+2\phi_{23}
&=& 2s_1dt\we e_1+2(2c_1+1)e_{23}+2(2c_2+1)(2c_3+1)e_{23}\y
&=& 2s_1dt\we e_1 + (4+8c_2c_3)e_{23}\y
&=& 2s_1dt\we e_1 - 8s_2s_3e_{23}\y
&=& \half\w_1,\ea\]
which is consistent with \eqref{PW} and $\ka=1$.
\end{proof}

Next, consider the locally-defined 1-form \bel{f} f = \oa\,da =
\oa(\alp+a\phi) = \beta+\phi\ee on $\Si^-$ over a self-dual Einstein
4-manifold. It is vector valued on the 3-sphere bundle, and transforms
as $f\lto\oq fq$. It arises from the Maurer-Cartan form of the fibre,
and
\[ df = d(\oa\,da)=d\oa\we da=(d\oa\,a)\we(\oa\,da)=-f\we f\]
so in components, \bel{fff} df_1=-2f_{23},\quad df_2=-2f_{31},\quad
f_3=-2f_{12}.\ee Regard $S^7=\tP^-\!S^4$ as the $S^3$ bundle over $S^4$.
Consider the 1-forms $(e_k)$ and $(f_k)$ satisfying \eqref{fff} and
\eqref{eee}. Lemma \ref{trick} yields\vs

\begin{theorem}\label{conn} For $k=1,2,3,$ we have
$\beta_k = (2c_k+1)e_k+f_k$.
\end{theorem} 

\noindent These equations relate the 1-forms from Lemma 1.1 to the
action of $SO(4),$ which is a significant step in understanding the
cohomogeneity-one setup. The trigonometric terms arise from the action
of the Weyl group in the cohomogeneity-one theory \cite{GWZ}, and is
suggestive of a discrete Fourier transform. One cannot usefully
substitute $(\beta_k)$ by the 3-Sasakian 1-forms $(\eta_k)$ in the
theorem, since this would yield $a\phi\oa$ in \eqref{f}.\vs

An orthonormal basis of 1-forms for the round metric $g_7$ on $S^7$ is
given by
\[dt,\ \ 2\beta_1,\ \ 2\beta_2,\ \ 2\beta_3,\ \
4s_1e_1,\ \ 4s_2e_2,\ \ 4s_3e_3.\]
Set $\tau=t/2$. Then

\[\ba{rcl} g_7
&=& dt^2 + 4\sul_{k=1}^3(\beta_k)^2 + \pi^*g_4\yyy
&=& dt^2 + 4\sul_{k=1}^3(\beta_k)^2 + 16\sul_{k=1}^3(s_ke_k)^2\yyy
&=& dt^2 + 4\sul_{k=1}^3\big[f_k^2+\ga_k(e_kf_k+f_ke_k)+\de_ke_k^2\big],\ea\]
\vskip-10pt where
\[\ga_k\= 2c_k+1 = -1+4\cos^2\tau,\quad
\de_k\= (2c_k+1)^2+(2s_k)^2 = 1+8\cos^2\tau.\]
It follows that the round metric can be
expressed as $dt^2$ plus a $6\times6$ matrix with $2\times2$ symmetric
blocks, of which the first is
\bel{block} 4\left(\ba{ccc} 1 && -1\+4\cos^2\tau\yy
-1\+4\cos^2\tau && 1\+8\cos^2\tau\ea\right).\ee
This can be reconciled with \cite[p\,7]{Z} up to scale, in which the
period of $t$ in \eqref{pcs} has been halved to $\pi/6$. A change of
off-diagonal sign results from our ASD convention.\vs

We now turn to the squashed metric \eqref{gt7} on $S^7$ and its \2
structure. Consider the orthonomal basis 1-forms $(\F_i)$ defined by
$\F_{2k-1}=\la\beta_k,$ $\F_{2k}=4s_ke_k,$ and $\F_7=-dt$. Theorem
\ref{conn} implies

\begin{cor}\label{FF} Let $\ka=1$ and $\la=2/\5$. Then
\[\left(\ns\ba{c} \F_{2k-1}\y \F_{2k}\ea\ns\right) \=
 \left(\ba{cc} \la & \la(2c_k+1)\yy 0 & 4s_k\ea\right)
 \left(\ns\ba{c} f_k\y e_k\ea\ns\right)\] for $k=1,2,3$.
\end{cor}

\noindent The behaviour of both $g_7$ and $\tg_7$ at the two singular
orbits of $SO(4)$ can be investigated by setting $t=0$ and $t=\pi/3,$
so that $\F_2=0$ and $\F_4=0,$ respectively.\vs

Referring to \eqref{varphi},
we now see that $\la^3\Up=\F_{135}$ and
\[2\la\Th \= (\F_{72}-\F_{46})\we \F_1+(\F_{74}-\F_{62})\we \F_3
+(\F_{76}-\F_{24})\we \F_5.\]
Proposition \ref{2.1} therefore yields\vs

\begin{theorem} The natural \2 structure on $S^7$ has
\[ \vp \= dt\we (\F_{12}+\F_{34}+\F_{56}) +
\F_{135}-\F_{146}-\F_{236}-\F_{245}.\]
\end{theorem}

The theorem asserts that in the stated basis, $\vp$ assumes a standard
form
\[\ba{rcl}
\vp &=& dt\we \xi + \Re\Xi\y
*\vp &=& -dt\we \Im\Xi +\frac12\xi^2\ea\]
relative to the $SO(4)$ orbits, where
\bel{xiX}\ba{rcl}
 \xi &=& \F_{12}+\F_{34}+\F_{56}\y
 \Xi &=& (\F_1+i\F_2)\we(\F_3+i\F_4)\we(\F_5+i\F_6).\ea\ee

To sum up: we have described the `squashed' 3-form $\vp$ from
Propositions \ref{2.1} and \ref{2.2} in terms of $SO(4)$ invariant
1-forms, via Corollary \ref{FF}. This provides a model for the
construction of \2 structures in the cohomogeneity-one situation of an
interval times (a finite quotient of) $SO(4)$. In the next section, we
shall discover a strikingly similar description on the principal part
of the isotropy irreducible space $SO(5)/SO(3)$.

\section{The cohomogeneity-one action of SO(4) on B$^7$}

Berger discovered two homogeneous spaces of positive sectional
curvature, of respective dimensions 7 and 13, namely
\[ B^7=\frac{SO(5)}{SO(3)},\qquad
B^{13}=\frac{SU(5)}{Sp(2)U(1)},\] which he studied in
\cite{B}. The former is the subject of this section; it is the coset
space that parametrizes principal subalgebras $\so(3)$ of $\so(5)$,
equivalently Veronese surfaces in $S^4$ \cite{BM1}.

The subgroup $SO(3)$ of $SO(5)$ acts irreducibly on $\R^5,$ which can be
identified with the space $S^2_0(\R^3),$ equivalently the space of
traceless symmetric $3\times3$ matrices. The isotropy subgroup of
$B^7$ is a principal 3-dimensional subgroup that acts irreducibly on
$\R^7$. Indeed,
\bel{53}\so(5)=\so(3)\op\fm,\ee
where $\fm$ can be identified with the space $S^3_0(\R^3)$ of primitive
cubic forms on $\R^3$. After complexification,
\bel{so5}
\so(3)_c\con S^2(\C^2),\quad
(\R^5)_c\con S^4(\C^2),\quad
\fm_c   \con S^6(\C^2)\ee
where $\C^2$ denotes the faithful representation of the double-cover
$SU(2)$ of $SO(3)$.

It is known that the action of $SO(3)$ on $\fm$ factors through
$G_2$. Indeed, there is an $SU(2)$ equivariant isomorphism \bel{864}
\La^3(S^6(\C^2))\con S^{12}\op S^8\op S^6\op S^4\op S^0\ee (with
obvious abbreviation), and the trivial 1-dimensional submodule
$S^0\cong\C$ is spanned by a $G_2$-admissible 3-form $\vp$. The latter
also arises from the projection to $\fm$ of the Lie bracket
$[\fm,\fm]$ within \eqref{53}. The fact that $[\fm,\fm]$ does not lie
in $\so(3)$ implies that $B^7$ is not a symmetric space, and its
holonomy group equals $SO(7)$.

Now, $\vp$ determines an $SO(5)$-invariant 3-form on $B^7$ (which we
denote with the same symbol), and an invariant 4-form again unique up
to scaling (which we can denote $*\vp$). Since there is no non-zero
invariant 2-form, $d\vp$ must be a nonzero multiple of $*\vp$, which
is the \2 condition \eqref{NP}. Hence, the well-known\vs

\begin{prop}
  $B^7$ has an $SO(5)$-invariant \2 structure.
\end{prop}
 
\noindent It follows that the Riemannian cone $B^7\times\R^+$ has
holonomy $\spin(7)$. It is not known whether this singularity can be
resolved in the manner of \cite{Baz}, but it motivates the study of
metrics with holonomy $\spin(7)$ admitting a codimension-two action by
$SO(4)$.

\begin{rem}\rm
The space $B^{13}$ is not isotropy irreducible. There is a fibration
$B^{13}\to\CP^4,$ with fibre $SU(4)/Sp(2)\cong S^5$. Its complexified
isotropy representation is $(V\op\ol V)\op\La^2_0E,$ where $V = E\ot
L\cong\C^4$ is a standard representation of $Sp(2)\times U(1)$. The
space $B^{13}$ has a 3-dimensional space of invariant 4-forms, which
are relevant to the study of `strongly positive curvature' \cite{ADS}.
\end{rem} 

The 3-form $\vp$ on $B^7$ was computed explicitly in \cite{GKS} and in
\cite{BM1}, but in different ways. The calculations below are
equivalent, but our approach emphasizes the underlying 3-fold
symmetry.

Denote by $S_{ij}$ the symmetric $3\times 3$ matrix with $(i,j)$ and
$(j,i)$ entry equal to $1$ and $0$ elsewhere. We then adopt the
following basis of traceless symmetric matrices: \bel{oS}\ba{c}
\oS_1=\frac1{\sqrt6}\diag(1,1,-2),\qquad
\oS_2=\frac1{\sqrt2}\diag(1,-1,0),\yy \oS_3=S_{12},\quad
\oS_4=S_{13},\quad \oS_5=S_{23}.\ea\ee We represent elements of the
Lie algebra \eqref{so5} by skew-symmetric $5\times5$ matrices.  Let
$E_{ij}$ denote the matrix with $1$ at $(i,j),$ $-1$ at $(j,i),$ and
$0$ elsewhere. For the subalgebra $\so(3),$ we may choose a basis
$\{\vep_1,\vep_2,\vep_3\},$ where
\[\ba{rcl}
\vep_1 &=& \6(2E_{23}+E_{45})\\
\vep_2 &=& \6(E_{24}+E_{35}+\sqrt3E_{14})\\
\vep_3 &=& \6(-E_{25}+E_{34}+\sqrt3E_{15}).\ea\]
One can check that $\tr(\vep_i\vep_j)=-2\delta_{ij},$ and that
$[\vep_i\vep_j]=-\la\vep_k,$ where $(ijk)=(123)$.

Now set
\[ R(t)=\left(\!\ba{ccccc}
\cos t&\sin t&0&0&0\\
-\sin t&\cos t&0&0&0\\
0&0&0&1&0\\
0&0&0&0&1\\
0&0&1&0&0\ea\!\right).\]
We have chosen the lower block to be a $3\times3$ permutation matrix
(rather than the identity) because $\rho=R(\frac{2\pi}3)$ has special
significance; it is derived from the effect
on \eqref{oS} of permuting the basis of $\R^3$. It therefore lies in
the intersection of the principal $SO(3)$ with the subgroup
$SO(2)\times SO(3)$ of $SO(5)$. Conjugation by this element of order 3
cycles the $\vep_i,$ so that $\rho^{-1}\vep_1\rho=\vep_2$, etc.

Let $a=\cos(2\pi/3)$ and $b=\sin(2\pi/3)$.
A basis for $\fm$ is given by $\{\ga_1,\ldots,\ga_7\}$ where
\bel{ga1-7}\ba{rcl}
\ga_1 &=& \6(E_{23}-2E_{45})\\
\ga_2 &=& -E_{13}\\
\ga_3 &=& \6(bE_{15}+aE_{25}-2E_{34})\\
\ga_4 &=& -aE_{15} +bE_{25}\\
\ga_5 &=& \6(-bE_{14}+aE_{24}+2E_{35})\\
\ga_6 &=& -aE_{14}-bE_{24}\\
\ga_7 &=& E_{12}.\ea\ee
One can check that $\tr(\ga_i\ga_j)=-2\delta_{ij},$ and that
$\tr(\ga_i\vep_j)=0$ for all $i,j$. Observe that 
$\rho$ cycles the pairs
$(\ga_{2k-1},\ga_{2k})$ for $k=1,2,3,$ in the sense that
$\rho^{-1}(\ga_1,\ga_2)\rho=(\ga_3,\ga_4)$ etc. In fact, it was only
necessary to choose the orthogonal basis $\ga_1,\ga_2$ orthogonal to
$\so(3)$; the rest were generated by $\rho$.
On the other hand, $\ga_7$ commutes with $\rho$.

Having made these choices, the $SO(5)$-invariant 3-form $\vp$ must
be proportional to
\[{\ts\frac1{10}}
\sum_{i<j<k}\tr([\ga_i,\ga_j]\ga_k)\>\ga_i^*\we\ga_j^*\we\ga_k^*,\]
where $(\ga_i^*)$ is the dual basis of invariant 1-forms.
With the given normalization, we discover\vs

\begin{prop}\label{127}
$\vp \=
\ga^*_{127}+\ga^*_{135}-\ga^*_{146}-\ga^*_{236}
-\ga^*_{245}+\ga^*_{347}+\ga^*_{567}$.
\end{prop}

\noindent Thus is in canonical form for a $G_2$ structure. Since $B^7$
has a unique $SO(5)$-invariant Riemannian metric $g_B$ (up to an
overall constant), it is a corollary of the proposition that
$(\ga_i^*)$ is an orthonormal basis for $g_B$.\vs

There is a cohomogeneity-one action of $SO(4)$ on $B^7$ \cite{PV}. The
principal orbits are $SO(4)/\Z_2^2$, and the singular orbits are both
isomorphic to $SO(4)/O(2)$. To analyse this, we start with bases
$(e_k)$, $(f_k)$, whose exterior derivatives are normalized as in
\eqref{eee}, \eqref{fff}. We then define positive and negative
combinations $\U_k=\half(f_k+e_k)$ and $\V_k=\half(-f_k+e_k),$ so that
\bel{VU} \left(\ns\ba{c} \U_k\y\V_k\ea\ns\right) \= L\left(\ns\ba{c}
f_k\y e_k\ea\ns\right),\qquad L= \half\left(\ns\ba{cc}
1&-1\y1&1\ea\ns\right)\ee for $k=1,2,3$. This switch of bases will
play a significant role in simplifying subsequent expressions, as we
are about to see. A coefficient of $1/2$ is chosen in the light of the
differential relations \eqref{dUV} below.

The Maurer-Cartan form
of the subgroup $SO(4)$ in $SO(5)$ can be expressed as
\bel{MC} A = 2
\left(\!\ba{ccccc}
0&0&0&0&0\\ 0&0&\V_3&\V_2&\V_1\\ 0&-\V_3&0&-\U_1&\U_2\\
0&-\V_2&\U_1&0&-\U_3\\ 0&-\V_1&-\U_2&\U_3&0,\ea\!\right)\ee
so that $e_1$ corresponds to the combination $-E_{25}-E_{34}$, and
$f_1$ to $E_{25}-E_{34}$, etc. These choices match the conventions
\eqref{eee},\eqref{fff} and (the subsequent) \eqref{dUV}.

The pullback of $A$ on $[0,\frac\pi3]\times SO(4)$ is given by
\[ R(t)^{-1}\!AR(t) + \ga_7^*\we dt=
\left(\!\ba{ccccc}
0 & dt  & -s\V_1 & -s\V_3 & -s\V_2\yy
-dt & 0 & c\V_1 & c\V_3 & c\V_2\yy
s\V_1 & -c\V_1 & 0 & -\U_2 & \U_3\yy
s\V_2 & -c\V_3  & \U_2  & 0 & -\U_1\yy
s\V_3 & -c\V_2 & -\U_3 & \U_1 & 0\ea\!\right)\]
where $c=\cos t$ and $s=\sin t$ as temporary abbreviations.

Let $\G_i$ denote the pullback of $2\ga_i^*$. These 1-forms satisfy
\[\ba{rcl}
\G_1 &=& \la(c\V_1+2\U_1)\\
\G_2 &=& 2s\V_1\\
\G_3 &=& \la((ca-sb)\V_2+2\U_2)\y
\G_4 &=& 2(sa+cb)\V_2\y
\G_5 &=& \la((ca+sb)\V_3+2\U_3)\y
\G_6 &=& 2(sa-cb)\V_3\y
\G_7 &=& 2dt
\ea\]
where $\la=2/\5$ and (as before) $a=\cos(2\pi/3)$ and $b=\sin(2\pi/3)$.

These equations manifest the $\Z_3$ symmetry arising from the action
of the Weyl group for the action of $SO(4)$ \cite{GWZ}, and can be
restated as\vs

\begin{cor}\label{GG} Let $\ka=1$ and $\la=2/\5$. Then
\be
\left(\ns\ba{c} \G_{2k-1}\y \G_{2k}\ea\ns\right) \=
\left(\ba{cc} 2\la & \la c_k \yy 0 & 2s_k\ea\right)
\left(\ns\ba{c} \U_k\y\V_k\ea\ns\right)\ee
for $k=1,2,3$.
\end{cor}

The 3-form from Proposition \ref{127} defining the \2 structure on
$B^7$ can be written \[\vp\= \xi\we \G_7 + \Re\Xi,\] where $\Xi$ is
defined by \eqref{xiX} with $\G_i$ in place of $\F_i$. If we denote
the $2\times2$ matrix in Corollary \ref{GG} by $M_k=M_k(t),$ the
analogue of \eqref{block} for the term of $g_B$ involving the plane
$\left<f_1,e_1\right>$ equals \bel{LM} L^\top\!M_1^\top\!M_1L =
\frac15\left(\ns\ba{ccc} 5\+4\sin^2t\+4\cos t && 1\-4\cos^2t\yy
1\-4\cos^2t && 5\+4\sin^2t\-4\cos t\ea\ns\right).\ee The change of
basis \eqref{VU} reveals a striking similarity between Corollary
\ref{FF} and Corollary \ref{GG}, which we explore next.

The 1-forms in \eqref{VU} satisfy
\bel{dUV}\left\{\ba{rcl} d\U_k &=&
-2(\U_{ij}+\V_{ij})\y d\V_k &=& -2(\U_i\V_j-\U_j\V_i)\ea\right.\ee
for $(ijk)=(123)$. If $\na$ is a connection on $T^*SO(4)$ that
satisfies
\[\na \V_k \= \U_i\ot \V_j-\U_j\ot \V_i\]
then we may regard the $\U_i$ as connection 1-forms on $\cD$,
whose curvature 2-forms are
\[ d\U_k+2\U_{ij}\= -2\V_{ij}.\]
The second equation of \eqref{dUV} implies that
\(\cD=\left<\G_2,\G_4,\G_6\right>\) is a differential ideal in
$SO(4)$. Regarding $\cD$ as a subspace of $\so(4)^*,$ its annihilator
is the diagonal subalgebra
\[ \so(3)_d\>\subset\>\so(3)_++\so(3)_-=\so(4)\]
represented by the bottom $3\times3$ block in \eqref{MC}.  This
subalgebra corresponds to the choice of isometry \(\si\colon
\left<e_1,e_2,e_3\right>\to\left<f_1,f_2,f_3\right>\) defined by our
indexing, or (up to sign) the vector in $\R^4$ that is fixed by our
choice of geodesic.

Inside each generic $SO(4)$ orbit, the integral submanifolds of $\cD$
consist of $\Z_2^2$ quotients of conjugacy classes of a diagonal
$SO(3)_d$, and are parametrized by $SO(4)/SO(3)_d\cong S^3$. Moreover,
\[\Im\Xi=-Y_{246}+Y_{235}+Y_{145}+Y_{136}\]
vanishes on these submanifolds, and they are associative (see
\eqref{xiX}). They play a role analogous to that of the orbits of
$SU(2)^-$ acting on $S^7$, insofar as the homogeneous \2 metrics are
concerned. Note that the complete theory of associative submanifolds
of $B^7$ is very rich because of the many different orbit types of a
principal $SO(3)$ subgroup of $G_2$ acting on the space $G_2/SO(4)$ of
associative subspaces of $\R^7$ \cite{BM1}.

\begin{rem}\rm
  In the wider homogeneous context, conjugacy classes of 3-dimensional
  subalgebras of $\so(5)$ (namely $\su(2)^-$, $\so(3)_d$ and principal
  $\so(3)$) correspond to critical submanifolds (namely $S^4$,
  $\Gr_3(\so(5))$ and $B^7$) for the natural Morse-Bott functional on
  the 21-dimensional Grassmannian $\Gr_3(\so(5))$. Two associated
  unstable manifolds of dimensions 8 and 12 then admit incomplete
  quaternion-K\"ahler metrics by the theory in \cite{KS}, and it is
  tempting to tie this in with the results of this article.
\end{rem}

The behaviour of the 1-forms $Y_i$ at $t=0$ and $\pi/3$ gives some
information about the singular locus and $SO(4)$ action. For
example, $Y_2$ vanishes at $t=0$, so the stabilizer $O(2)$ is tangent
to the annhilator of $\left<Y_1,Y_3,Y_4,Y_5,Y_6\right>$. At our
distinguished point, $SO(2)$ rotates $\left<E_4,E_5\right>$ by
$\theta$ and $\left<E_2,E_3\right>$ by $2\theta$. Its induced action
on \eqref{+-} fixes $E_{23}\pm E_{45}$, but rotates
$\left<E_{24}+E_{53},E_{25}+E_{34}\right>$ by $3\theta$, whereas
$\left<E_{24}-E_{53},E_{25}-E_{34}\right>$ by $-\theta$. The singular
orbit $SO(4)/O(2)$ is therefore covered by the circle bundle in
$\cO(3,1)\to\CP^1\times\CP^1$.
  
The subgroup $SU(2)^+$ of $SO(4)$ (defined by our choice of oriented
basis) has stabilizer $\Z_3$ at $t=0$, but $SU(2)^-$ acts freely. The
opposite is true at $t=\pi/3$, giving rise to a group diagram with
weights $(3,1)$ at $t=0$, and $(1,3)$ at $t=\pi/3$, relative to
$SO(4)=SO(3)^+SO(3)^-$. It follows that $B^7/SU(2)^+$ is smooth for
$t\in(0,\pi/3]$, but not at $t=0$. Since $SO(4)/SU(2)^+\cong SO(3)$,
  the commutator $SU(2)^-$ acts on the quotient with generic
  stabilizer $V=\Z_2^2$. This observation leads to the
  $SO(4)$-equivariant diffeomorphism $B^7\cong\tP^+\kr\cO_5$
  established by Grove, Wilking, and Ziller \cite{GWZ}, where $\cO_k$
  is Hitchin's orbifold with a $\Z_{k-2}$ singularity transverse to
  $\RP^2$ (we use $k\ge3$ as in \cite{H}). We outline this setup next
  because of its implications for \2 structures with an $SO(4)$
  symmetry, working backwards from $\cO_k$. Recall our convention that
  a self-dual conformal structure has vanishing Weyl component $W_-$
  (as for $\CP^2$ whose K\"ahler 2-form $\omega$ lies in $\La^+$).

The irreducible action of $SO(3)=G$ on $S^4$ has generic orbit $G/V$,
and singular orbits $\cL\cong G/L$ and $\cR\cong G/R$ where $L$ and
$R$ are copies of $O(2)$ in different blocks of
$SO(3)$. Topologically, $\cO_k$ is constructed by gluing disk bundles
over $\cL$ and $\cR$, each a copy of $\RP^2$. If we choose $\cL$ to be
the singular locus, then
\[\cO_k\ \approx\ (G\times\sb{L/\ell}(D^2/\ell)\bigcup_{G/V}\>
(G\times\sb RD^2),\] and this space is still homeomorphic to $S^4$.
The subgroups $L,R$ act on the tangent and normal spaces $T,T^\perp$
as $(e^{\ell i\theta},e^{-2i\theta})$ at $x\in\cL$, and as
$(e^{i\theta},e^{2i\theta})$ at $y\in\cR$. The factors of $2$ are
needed to match the generic stabilizer $V$. The actions on $\La^+(T\op
T^\perp)$ are therefore $e^{(\ell-2)i\theta}$ at $x$ and
$e^{3i\theta}$ at $y$. The inclusion \bel{wts} U(1)\subset L\subset
SO(3)\times SU(2)^+SU(2)^-\ee is given by $(e^{\ell i\theta},
e^{(\ell-2)i\theta},e^{(\ell+2)i\theta})$ at $x$, and the differing
weights $\ell$, $\ell+2$ at $x$ show that the $SO(3)^+$ principal
bundle $P^+\kr\cO_k$ is smooth. The same is true of $P^-\kr\cO_k$,
though the SD/ASD distinction can be verified by considering $\CP^2$
for which $\La^+=\left<\omega\right>\op[\![\La^{2,0}]\!]$ extends the
canonical line bundle. The singular locus $\cL$ of $\cO_4=\CP^2/\Z_2$
is totally real, reflecting the weight $\ell-2=0$. When $\ell=3$, the
group diagram for $\tP^+\kr\cO_5$ with weights $(3,1)$ from
\eqref{wts} and (at $\cR$) $(1,3)$ coincides with that of $B^7$, if we
ignore the third factor in \eqref{wts}, and let $SO(3)$ play the role
of $SU(2)^-$ on $B^7$. Whence the equivalence $\tP^+\kr\cO_5\cong
B^7$.

\begin{rem}\rm Although $B^7$ cannot be the total space of a smooth
principal $S^3$ bundle, it is known to be diffeomorphic to a smooth
$S^3$ bundle over $S^4$ with Euler class $\pm10$ and (in particular)
$p_1=\pm16$ \cite{GKS}.\end{rem}

The orbifolds $\cO_k$ arise from those on the moduli spaces of
monopoles of charge 2 and mass 1 on hyperbolic space
$\cH^3(-16/(k-2)^2)$ for $k\ge3$. Their metrics can (in theory, all)
be constructed from solutions of Painlev\'e VI, and approach the
hyperk\"ahler metric on the Atiyah-Hitchin space $M^2_0$ \cite{H,AH}
as $k\to\8$. Explicit algebraic expressions and fascinating graphs for
$\ell=3,4,6$ can be found in \cite{Z}. The Konishi bundle
$P^-\kr\cO_k$ inherits an $SO(4)$-invariant 3-Sasakian
metric. Therefore both $P^-\kr\cO_5$ and $P^+\cO_5$ admit (albeit,
very different) \2 structures, and it is natural to ask whether the
same could be true for $k>5$. There is an analogous question for
positive sectional curvature, since modification of the 3-Sasakian
metric on $P^-\kr\cO_5$ has this property for $k=5$ \cite{GWZ,De2}.

\section{Hypersurface calculations}

In the setup of \eqref{xiX}, we resort to the description of an \2
manifold by means of the induced $SU(3)$ structures on the
hypersurfaces for which $t\in(0,\pi/3)$ is constant, so the generic
$SO(4)$ orbits. The condition \eqref{NP} is equivalent to the next two
equations: \bel{sys1}\ba{rcl} d\xi-\dt\Re\Xi &=& \npc\,\Im\Xi\ea\ee
\vspace{-20pt}
\bel{sys2}\ba{rcl}
d(\Re\Xi) &=& \npc\,\half\xi^2.\ea\ee
The second characterizes the so-called nearly-half-flat (NHF)
condition introduced by \cite{FIMU}, and studied in the present context by
the second author \cite{Si}.

Recall the \2 structure on $S^7$ determined by the the basis $(\F_i)$
and the $2\times2$ matrices that feature in Corollary \ref{FF}. Let us
just consider 6-dimensional orbits of $SO(4)$ with slightly more
general $SU(3)$ structures. The next result shows that it is to some
extent possible to reverse engineer the constants occurring in
Corollary \ref{FF} by merely applying the NHF condition for all values
of $t$:\vs

\begin{prop}\label{nhf} Define $\xi$ and $\Xi$ by \eqref{xiX} with
\[\left(\ns\ba{c} \F_{2k-1}\y \F_{2k}\ea\ns\right) \= \left(\ba{cc} \la
& \la(ac_k+b)\yy 0 & 4s_k\ea\right) \left(\ns\ba{c} f_k\y
e_k\ea\ns\right),\] where $\la,a,b$ are (unconstrained)
constants. Then \eqref{sys2} holds for all $t$ iff
\begin{enumerate}[itemsep=0pt,topsep=-6pt]
  \item[\rm(i)] $(a,b)=(0,-1)$ or $(0,0),$ in which case $\npc=-2/\la,$ or
  \item[\rm(ii)] $(a,b)=(\pm2,1)$, in which case $\npc=-\frac12(\la^2+4)/\la$.
\end{enumerate}
\end{prop}

\noindent In this formulation, the coefficient `4' is retained as a
normalization.\vs

\begin{proof} Equations \eqref{xiX} yield 
\( \xi\= -4\la(s_1e_1f_1+s_2e_2f_2+s_3e_3f_3)\)
so
\[\half\xi^2 \=
-16\la^2(s_2s_3e_{23}f_{23}+s_3s_1e_{31}f_{31}+s_1s_2e_{12}f_{12}).\]
Moreover
\[\ba{rcl} d\,\Xi &=&\ds
\left(-2\la f_{23}-2\la(ac_1\+b)e_{23}-8is_1e_{23}\right)\we\y
&&\ds\phantom{mmm} (\la f_2+\la(ac_2\+b)e_2+4is_2e_2)\we
(\la f_3+\la(ac_3\+b)e_3+4is_3e_3)+\cdots\ea\]
Therefore the $e_{23}f_{23}$ component, call it $\Ga_1,$ of
$d(\Re\Xi)$ equals
\[\ba{rcl} 
&& -2\la[\la^2(ac_2\+b)(ac_3\+b)-16s_2s_3]-2\la^3(ac_1\+b)\y
&=& -2\la^3[a^2c_2c_3-abc_1+b^2]+32\la s_2s_3-
  2\la^3(ac_1+b)\y
&=& -2\la^3(b+b^2)-2\la^3a^2c_2c_3+32\la s_2s_3+2\la^3(ab-a)c_1\y
  &=& -2\la^3(b+b^2)-16\la-2\la(16+a^2\la^2)c_2c_3+2\la^3(ab-a)c_1,\ea\]
the last line by \eqref{ccss}.
This must equal $\mu$ times
$-16\la^2s_2s_3=8\la^2(1+2c_2c_3)$ for all $t$.
To eliminate the coefficient of $c_1$ above, we need $a=0$ or $b=1$.

If $a=0,$ we obtain the system

\[\left\{\ba{l} -32\la=16\la^2\npc\y
-2\la^3b(1+b)-16\la=8\la^2\npc,\ea\right.\] which implies that
$\la\mu=-2$ and $b=0$ or $-1$. 
If $b=1,$ we obtain
\[-4\la^3-2\la(8+a^2\la^2)c_2c_3-8\la\= \npc(8a^2\la^2+16c_2c_3),\]
which provides the system
\[\left\{\ba{l} \la^2+2\npc\la+4=0\y a^2\la^2+2\npc\la+4=0,\ea\right.\]
and $a=\pm2$. The unique quadratic then gives the constraint between
$\la$ and $\npc$.
\end{proof}

\noindent As a corollary, any positive definite linear combination of
$g_7$ and $\tg_7$ on $S^7$ is associated to an NHF structure. Note
however that $\la=2/\5$ implies that $\npc=-6/\5,$ as in Proposition
\ref{2.1} modulo orientation.\vs

The analysis of \eqref{sys1} requires more extensive
calculation. Turning to $B^7,$ we have established that the solution
in Corollary \ref{GG} is essentially unique within a 3-parameter
family. This is part (i) of
  
\begin{theorem}\label{main}
Define $\xi$ and $\Xi$ by \eqref{xiX} with $\G_i$ in place of
  $\F_i,$ where now
\[\left(\ns\ba{c} \G_{2k-1}\y \G_{2k}\ea\ns\right) \=
\left(\ba{cc} 2\la & 2\la a c_k \yy b & 2s_k\ea\right)
\left(\ns\ba{c}\U_k\y\V_k\ea\ns\right),\] and $\la,a,b$ are
(unconstrained) constants. Then
\begin{enumerate}[itemsep=0pt,topsep=0pt]
\item[\rm(i)] \eqref{sys1} and \eqref{sys2} both
  hold iff $(\la,a,b)=(\pm\frac2\5,\frac12,0)$;
\item[\rm(ii)]  \eqref{sys2} holds if either
  $\la^2(4-a^2)=3$ and $b=0$, or $(\la,a,b)=(\frac12,2,\pm\sqrt3)$.
\end{enumerate}
In all cases, $\mu$ is determined by $\la$.
\end{theorem}

\noindent We omit the proof of (i), since analysis of \eqref{sys1} is
somewhat fruitless. The proof of (ii) is analogous to that of
Proposition \ref{nhf}, but more involved since it is based upon
\eqref{dUV}. Both parts were checked by hand and using \emph{Maple},
but in different orders.\vs

The significance of the condition $b=0$ is that it guarantees smooth
extension to the singular orbits. This justifies our presentation of
the $S^7$ and $B^7$ solutions using upper triangular $2\times2$
matrices, and contrasting geometry at the singular orbits. The special
NHF solutions in this section will be analysed elsewhere.

There is an obvious generalization of the hypotheses of Theorem
\ref{main} that will also capture the \2 solution on the principal
part $SO(4)\times(0,\pi/3)$ of $B^7$. Namely one can saturate the
first row of the $2\times2$ matrix by linear combinations $1$ and
$c_k$, and its second row by terms proportional to $s_k$. For that
matter, one could insert less trivial Fourier series in each slot. The
prospect of adapting the methods to \cite{FH} to find new compact \2
manifolds of cohomogeneity one appears remote. Nonetheless, the
results of this section motivate a computational search for other \2
solutions of the above type.

\small

\section*{Acknowledgements}

Both authors were supported by the Simons Foundation (\#488635 Simon
Salamon). Some of the results of this paper were presented at the
conference `Special Holonomy and Geometric Structures on Complex
Manifolds' at IMPA in March 2024, and we thank the organizers.

\vs\vs\flushleft\footnotesize

University of M\"unster, 
Einsteinstrasse 62, 48149 M\"unster, Germany.\\
\texttt{rsinghal@uni-muenster.de}\yy

King's College London, Strand, London, WC2R 2LS, UK.\\
\texttt{simon.salamon@kcl.ac.uk}

\enddocument